\documentclass{amsart}

\usepackage[pdftex]{graphicx}
\usepackage{amssymb,amsfonts}
\usepackage{url}

\newcommand{\nc}{\newcommand}
\nc{\dmo}{\DeclareMathOperator}
\nc{\Q}{\mathbb Q}
\nc{\Z}{\mathbb Z}
\nc{\Qp}{\Q_p}
\nc{\Zp}{\Z_p}
\nc{\C}{{\mathbb C}}
\nc{\R}{{\mathbb R}}
\nc{\bbP}{{\mathbb P}}
\nc{\F}{{\mathbb F}}
\dmo{\Gal}{Gal}
\dmo{\PSL}{PSL}
\dmo{\PGL}{PGL}
\dmo{\GL}{GL}
\dmo{\SL}{SL}
\nc{\Fr}{\textrm{Fr}}
\nc{\g}{\PSL_2(7)}
\nc{\gseptic}{\GL_3(2)}
\nc{\rd}{\mathrm{rd}}
\nc{\unr}{\text{unr}}
\nc{\cots}{,\ldots, }
\nc{\pots}{+\cdots+ }
\nc{\sd}{{:}}
\nc{\z}{{,}}
\nc{\Se}{\S}
\nc{\into}{\hookrightarrow}
\nc{\frakd}{\mathfrak{d}}
\nc{\gaffine}{2^3\sd\gseptic}

\newtheorem{thm}{Theorem}
\newtheorem{cor}{Corollary}
\newtheorem{lemma}{Lemma}

\title{Mixed degree number field computations}

\author{John W.\ Jones}
\address{School of Mathematical and Statistical Sciences, Arizona
  State University, PO Box 871804, Tempe, AZ 85287} \email{jj@asu.edu}

\author{David P.\ Roberts}
\address{Division of Science and Mathematics, University of
  Minnesota-Morris, Morris, MN 56267}
\email{roberts@morris.umn.edu}
\thanks{Roberts was supported by grant \#209472 from the Simons foundation.}
\keywords{number field; discriminant}

\subjclass[2010]{Primary 11R21; Secondary 11Y40, 11R32}

\begin{document}
\begin{abstract}
  We present a method for computing complete lists of number fields in
  cases where the Galois group, as an abstract group, appears as a
  Galois group in smaller degree.  We apply this method to find the twenty-five 
  octic fields with Galois group $\g$ and smallest absolute discriminant. 
  We carry out a number of related computations, including determining
  the octic field with Galois group $\gaffine$ of smallest absolute discriminant. 
\end{abstract}
\maketitle

\section{Introduction}
\subsection{Overview}
Number theorists have computed number fields with minimal absolute
discriminants for each of the thirty possible Galois groups in degrees
at most $7$.  In degrees $8$ and $9$, the minimal fields are known for
the seventy-five solvable Galois groups.  All these minimal fields are
available, together with references to sources, on the Kl\"uners-Malle database
 \cite{kluner-malle}.  The minimal fields are also available, typically as first elements on long
complete lists, at our database \cite{jr-global-database}.

The Kl\"uners-Malle paper \cite{kluner-malle} also gives smallest known absolute discriminants 
for the five nonsolvable octic groups and the four nonsolvable nonic groups.  
Despite the fifteen years that have passed since its publication, rigorous
minima have not been established for these nine groups.  In this paper,
we address two of the nine cases, proving that the absolute discriminants $3^8 7^8$ and $5717^2$ 
presented in \cite{kluner-malle} for 
the octic groups $\g$ and $\gaffine$ are indeed minimal.  
These cases are related through the exceptional isomorphism 
$\g \cong \gseptic$.   

One element of our approach for finding the $\g$ minimum was suggested
already in \cite{kluner-malle}:  any octic $\g$ 
field $K_8$ has the same Galois closure as two septic 
$\gseptic$ fields $K_{7a}$ and $K_{7b}$.  As the discriminants
satisfy $D_{7a}=D_{7b} \mid D_8$, one 
can in principle establish minimality of the octic discriminant $21^8$ by 
conducting a search of all septic fields with absolute 
discriminant $\leq 21^8$. 
We combine this with the method of targeted Hunter searches, which
requires us to analyze, on a prime-by-prime basis, how discriminants
either stay the same or increase when one passes from septic to octic
fields.  This targeting based on discriminants  makes the computation
feasible.  We add several smaller refinements to make the computation
run even faster. 

Our title refers to the general method of carrying
out a carefully targeted search in one degree to  
obtain a complete list of fields
in a larger degree.  Section~\ref{background}
gives background and then Section~\ref{mixed} 
describes the general method,  using our
case where the two degrees are $7$ and $8$ as an illustration.  
Section~\ref{summary} presents our minimality result for $\g$, improved
in Theorem~\ref{octicthm1} to the complete list of twenty-five octic $\g$
fields with discriminant $\leq 30^8$.  
This section also presents corollaries giving minimal 
absolute discriminants for certain related groups 
in degrees $16$, $24$, and $32$.   

Section~\ref{pgl2}
gives a second illustration of the mixed degree method,
now with degrees $5$ and $6$.
Here we use the exceptional isomorphisms
$A_5 \cong \PSL_2(5)$ and $S_5 \cong \PGL_2(5)$
and Theorem~\ref{sexticthm} considerably extends the known list of 
sextic $\PSL_2(5)$ and $\PGL_2(5)$ fields.   
We also explain in this section potential connections
with asymptotic mass formulas and Artin representations.

Our final section returns to groups related to 
the septic group $\gseptic$.    Theorem~\ref{septicthm} 
finds all alternating septics with discriminant
$\leq 12^7$.   The long runtime of this
search makes clear the importance of targeting
for Theorem~\ref{octicthm1}.  
However just the bound $12^7$ is
sufficient for our last corollary, 
which confirms minimality
of the 
 $\gaffine$ field with
discriminant $5717^2$.

\subsection{Notation and conventions}
\label{notation}
We denote the cyclic group of order $n$ by $C_n$.  We use $N\sd H$ to
denote a semi-direct product with normal subgroup $N$ and complement $H$.

A number field is a finite extension of $\Q$, which we
consider up to isomorphism.  If $K/\Q$ is such an extension with
degree $n$, then its normal closure, $K^g$, is Galois over $\Q$.
Moreover, $\Gal(K^g/\Q)$ comes with a natural embedding into $S_n$,
which is well-defined up to conjugation.  We denote the image of such
an embedding, which is
a transitive subgroup of $S_n$,
by simply $\Gal(K)$.

Each possibility for $\Gal(K)$ has a standard notation
$nTj$ introduced in \cite{conway-hulpke-mckay}.
In \Se\ref{classnumbers} we use the classification of
nearly $3$ million transitive subgroups of $S_{32}$ which was
completed more recently in \cite{deg32}.

When several non-isomorphic fields have the same
splitting field, we refer to them as siblings. For
example, fields $K_{7a}$, $K_{7b}$, and $K_{8}$ as in the 
overview are siblings.  When choosing
notations for groups, we choose the name which
reflects the group's natural transitive action.
For example, $\gseptic=7T5$ and $\g=8T37$ act on the
projective spaces $\bbP^2(\F_2)$ 
and $\bbP^1(\F_7)$ of orders $7$ and $8$ respectively.   The third 
group mentioned in the overview is the
group of affine transformations of $\F_2^3$. 
Our notation emphasizes its semidirect product structure: $\gaffine = 8T48$.  

It is often
enlightening to shift the focus from an absolute discriminant
$|D|$ 
of a degree $n$ number field $K$
to the corresponding root discriminant 
$\rd(K)= \delta  = |D|^{1/n}$.     We generally try to indicate
both, as in the numbers $21^8$ and $30^8$ of the overview.  

\section{Background}
\label{background}
Our method of mixed degree targeted Hunter searches is built on well-established
methods of searching for number fields, which we now
briefly explain.  
\subsection{Hunter searches}
\label{hunter}
A Hunter search is a standard technique for computing all primitive number
fields of a given degree with absolute discriminant less than a given
bound \cite{cohen2}.   Here a number field $K$ is primitive if it has
exactly two subfields, itself and $\Q$.   When the degree $n$ is 
prime, as in our cases $n=7$ and $n=5$ here, the primitivity
assumption is vacuous. 

The only two inputs for a standard Hunter search are the
degree and the discriminant bound.  Some implementations optimize
for a particular signature which can then be thought of as a third
input, but here we search all signatures simultaneously.
The computation itself is an exhaustive search for polynomials with
integer coefficients bounded by various inequalities.

\subsection{Targeted Hunter searches}
Targeting, introduced in \cite{jr1}, and refined in
\cite{JR-septic}, allows one to search for fields with particular
large discriminants.  One carries out a Hunter search, but only for
fields which match a given combination of local targets.  The targets, 
described below,
determine both the discriminant and the local behavior of the field at
ramifying primes $p$.  This latter information forces a defining
polynomial to satisfy congruences modulo several prime powers, 
and these congruences greatly reduce the number of polynomials one
needs to inspect.

To describe the targets more precisely, let $p$ be a prime number and let 
$K$ be a degree $n$ number field.  Then $K\otimes\Q_p\cong \prod_{i=1}^g K_{p,i}$
where each $K_{p,i}$ is a finite extension of $\Q_p$.  At its most
refined level, a local target may be a single $p$-adic algebra, up to
isomorphism.  In a few situations, we do work at this level.
However typically, one wants to treat natural collections of
$p$-adic algebras as a single target.

Let $\Q_p^\unr$ be the unramified closure of $\Q_p$.  Then, similar to
the decomposition above we have
\begin{equation} \label{localdecomp}
K\otimes \Q_p^\unr\cong\prod_{j=1}^t L_{p,j},
\end{equation}
where each $L_{p,j}$ is a finite extension of $\Q_p^\unr$.  Let $e_j$
be the ramification index of the field $L_{p,j}$, and $(p)^{c_j}$ its
discriminant ideal.  We note that the $e_j$ give the sizes of the
orbits of the $p$-inertia subgroup acting on the roots of an
irreducible defining polynomial for $K$.
A typical local target is then a pair
$((e_1,\ldots,e_t), \sum_{j=1}^t c_j)$ with the list of $e_j$ weakly
decreasing.  The ramification indices $(e_1,\ldots,e_t)$ give a
partition of $[K:\Q]$, and the discriminant of the local algebra,
which equals the $p$-part of the discriminant of $K$, is
$(p)^{\sum_j c_j}$.

A local target at $p$ determines a list of congruences modulo some
fixed power of $p$.  For tamely ramified primes we work modulo $p$, while we use higher
powers for wildly ramified primes.  When there is more than one
ramifying prime, the lists of congruences are simply combined via the
Chinese remainder theorem.

\section{Mixed degree targeted Hunter searches}
\label{mixed}
In a mixed degree targeted search, one has a Galois group
$G$ and transitive permutation representations of two different degrees 
$n<m$.   Each degree $n$ field $K_n$ with Galois group
$G \into S_n$ determines a degree $m$ field $K_m$ 
with Galois group $G \into S_m$.  
Targets are triples $((e_1,\ldots,e_t), c_n, c_m)$ where 
$c_n$ is the local discriminant exponent for the small degree
fields searched, while $c_m$ is the local discriminant exponent
of the larger degree fields actually sought.  One uses the values
$p^{c_m}$ to decide which combinations of targets to search, 
and $((e_1,\dots,e_t),c_n)$ to carry out the actual
search in degree $n$.   
We describe
how one deals with the two degrees here, often
by using our first case with $n=7$ and $m=8$ as an example.   
Once we have the degree $n$ polynomials in hand, we compute the
corresponding degree $m$ polynomials as resolvents using {\em Magma}
\cite{magma}.

\subsection{Tame ramification}
\label{tame}  
The behavior of tame ramification under degree changes is straightforward.
Let $K$ be a degree $n$ number field, $G$ its Galois group, and $p$ a
tamely ramified prime.  The inertia subgroup $I$ for a prime above $p$
is cyclic; let $\sigma$ be a generator.
Via the given inclusion $G \subseteq S_n$, we let $e_1, e_2,\ldots, e_t$ be the cycle type
of $\sigma$.  These match the $e_j$ of the local target described
above.  The exponent of $p$ in the discriminant of $K$ is then given by 
\begin{equation}
 \label{tamec}
c_n = \sum_{j=1}^t (e_j-1) = n-t.
\end{equation}

When one is considering also a second degree $m$, one just
runs through the above procedure a second time.  In our first case, $G \cong \gseptic \cong \g$,
each row on Table~\ref{tame7table} represents a candidate for $\sigma$.  
The row then gives the corresponding pair of partitions $(\lambda_7,\lambda_8)$ and 
pair of discriminant exponents $(c_7,c_8)$. 
These discriminant exponents are computed
from the partitions via formula \eqref{tamec}.  
\begin{table}[htb]
\caption{Cycle types and discriminant exponents for $\gseptic \cong \g$ in
degrees $7$ and $8$. \label{tame7table}}
\[
\begin{array}{l|l||c|c}
\lambda_7& \lambda_8& c_7 & c_8 \\
\hline
7 & 71 & 6 & 6 \\
421 & 44 & 4 & 6 \\
331 & 3311 & 4 & 4 \\
22111 & 2222 & 2 & 4 
\end{array}
\]
\end{table}

Note that if a prime $p$ is tamely ramified in our pair of fields $(K_7,K_8)$,
 then its minimal contribution to the discriminant of the octic is $p^4$.  Thus,
when searching for octic fields with absolute discriminant $\leq B$,
we need only consider primes $p\leq \sqrt[4]{B}$.  Our largest
search used $B=30^8$, so $p\leq 900$.

The relation $D_7\mid D_8$ mentioned in the introduction is due to the
fact that we always have $c_7\leq c_8$, and that this inequality also holds for
wildly ramified primes.  In two of the tame cases, one has equality, but in the other two tame cases
one has strict inequality.    Our method using targeted searches makes
use of the strictness of these latter inequalities.  

\subsection{Wild ramification}
\label{wild}
An explicit description of the behavior of wild $p$-adic ramification under 
degree changes becomes rapidly more complicated as
$\mbox{ord}_p(G)$ increases.  We describe just our case $G \cong \gseptic \cong \g$ here,
as this case represents the basic nature of the general case well.  

Since $|G|=2^3\cdot 3\cdot 7$, the only primes which can be wildly
ramified in a $G$ extension are $2$, $3$, and $7$.
For a subgroup of $G$ to be an inertia group for a wildly ramified
prime $p$, it must be an extension of a cyclic group of order prime to
$p$ by a non-trivial $p$-group.  The candidates for a $G=\g$ extension
are given in Table~\ref{wildtable}.
\begin{table}[htb]
\caption{Wild ramification data for $\g$. \label{wildtable}}
\[ \begin{array}{c|c||c|cc|c||l} 
p &I &D & \multicolumn{2}{c|}{\lambda_7} & \lambda_8& \multicolumn{1}{c}{ (c_7, c_8)} \\
\hline\hline
7 &C_7\sd C_3 &C_7\sd C_3& \multicolumn{2}{c|}{7} & 71 & (8,8), (10,10)  \\
&C_7 & C_7,C_7\sd C_3& \multicolumn{2}{c|}{7} & 71 & (12,12) \\
\hline\hline
3&S_3 & S_3& \multicolumn{2}{c|}{331} & 62&  (6,8),(10,12) \\
&C_3 &C_3,S_3&  \multicolumn{2}{c|}{331} & 3311 & (8,8)  \\
\hline\hline
2&A_4 &S_4& 43&61 & 44 &  (6,8),(10,16) \\
&D_4 &D_4& \multicolumn{2}{c|}{421} & 8 & (12,22),(14,24) \\
&C_4 &C_4,D_4& \multicolumn{2}{c|}{421} & 44 &(14,22) \\
&V &V,D_4,A_4 & 2221&4111 & 44  & (6,12),(8,16) \\
&C_2 &C_2,V,C_4& \multicolumn{2}{c|}{22111} & 2222  & (4,8),(6,12)
\end{array}
\]
\end{table}
They run over all of the non-trivial proper subgroups of $\g$ up to
conjugation, with the exception of two conjugacy classes of subgroups
isomorphic to $S_4$.

Each subgroup in the table is a candidate for being the inertia group
for a wild prime for only one prime.  The horizontal lines
separate the subgroups according to this prime.
The second column gives the isomorphism type of the candidate for
inertia, and the third column gives corresponding candidates for the
decomposition group.  Over other $2$-adic ground fields, $A_4=I=D$ 
is possible, but not over $\Q_2$ since there is no ramified $C_3$
extension of $\Q_2$.

The columns labeled $\lambda_7$ and $\lambda_8$ show the orbit sizes of the
actions of $I$ in the degree seven and eight representations respectively.  There are two
possibilities for $A_4$ and $V$ in degree $7$, so we give both.  
The orbit sizes are helpful in
determining the data $c_7$ and $c_8$ in each case, and $\lambda_7$ is 
the partition of $7$ needed for carrying out the targeted Hunter search.

In most cases, it is clear from Galois theory how to interpret the
orbit sizes.  For example, inside a Galois $A_4$ field, there are
unique subfields of degrees $3$ and $4$ up to isomorphism.  So the $4$
in the first $A_4$ entry is for the usual quartic representation, and
the $3$ is its resolvent cubic.  More detailed computations with the
groups allow us to resolve the two ambiguities, which are as follows.
\begin{itemize}
\item A Galois $D_4$ field has three quadratic subfields and three
  quartic subfields (up to isomorphism).  In the degree $7$ partition
  $421$, the $4$ represents a quartic stem field, say defined by a
  polynomial $f$, and then the $2$ represents the field obtained from
  a root of $x^2-\textrm{Disc}(f)$.
\item A $V$ field has three quadratic subfields.  In the line
  for $V$, the $2221$ represents the product of these quadratic fields
  and $\Q_2$. 
\end{itemize}

The last column gives a list of candidate pairs $(c_7,c_8)$,
coming by analyzing the corresponding local extensions.  Some cases
can be done using just Galois theory and general properties
of extensions of local fields. A simple approach, however, which applies 
to all cases is to make use of the complete lists of the relevant local
fields \cite{jr-local-database,lmfdb}.
For example, suppose $2$ is wildly ramified with inertia subgroup
isomorphic to $A_4$.  Then, the decomposition group is $S_4$ and 
there are three possibilities for a Galois $S_4$ field with $A_4$ as its inertia
subgroup.  It is then a matter of checking the discriminants of the
relevant fields to complete the $A_4$ line.

The list of targets for each
prime is fairly straightforward to read off Table~\ref{wildtable}.  For
example, for $p=3$ we have only $((3,3,1), 6, 8)$, $((3,3,1), 8, 8)$, and
$((3,3), 10, 12)$.  With $p=2$, one has $((4,2,1),
14, 24)$ from $I=D_4$ and $((4,2,1), 14, 22)$ from $I=C_4$; however, only the
latter gets used since it has the same partition and $c_7$ and a
smaller value of $c_8$.

\subsection{Further savings}
    Various techniques can reduce the number of polynomials that it is 
necessary to search.   Again we illustrate these reductions by our
first case with $G \cong \gseptic \cong \g$.  The first and fourth reductions
below simply eliminate some local targets $((e_1,\dots,e_t),c_7,c_8)$
from consideration.  The second and third let us reduce
the size of some of the local targets.    

\subsubsection{Savings from evenness at half the tame primes}
Our first savings comes from $\gseptic$ being an even
subgroup of $S_7$, i.e., from the inclusion $\gseptic \subset A_7$.  
To obtain this savings, we make use of the following general lemma.
\begin{lemma}
  Suppose $n$ and $p$ are distinct primes, $K$ is a degree $n$ number 
  field whose Galois group $G$ is contained in $A_n$, and $p$ is totally 
  ramified in $K$.  Then $p$ is a quadratic residue modulo $n$.
\end{lemma}
\begin{proof}
  Let $D$ be the decomposition group for a prime above $p$.  Tame
  Galois groups over $\Q_p$ are $2$-generated by $\sigma$ and
  $\tau$ where $\sigma\tau\sigma^{-1}=\tau^p$ (see
  \cite{iwasawa-local-galois}).  Here, $\tau$ is a generator of the
  inertia subgroup and $\sigma$ is a lift of Frobenius.  Thus $D$
  isomorphic to $(\Z/n)\sd (\Z/f)$ where $f$ is the order $\sigma$,
  and the action of $\sigma$ on $\Z/n$ is multiplication by $p$.
  The Galois group locally is a subgroup of $A_n$ which forces
  $\sigma$ to be an even permutation, which in turn implies that
  $p$ is a square modulo $n$.
\end{proof}

\noindent In our situation, the lemma says that if $p\neq 7$ is totally ramified,
then $p$ must be congruent to $1$, $2$, or $4$ modulo $7$, eliminating
``totally ramified'' as a target for approximately half of the primes.

\subsubsection{Savings from evenness at $p=7$}
We can achieve a savings from the fact that $\gseptic \subset A_7$ at $p=7$ as
well. 
It is evident from the complete lists of degree $7$ extensions of
$\Q_7$ in \cite{jr-local-database, lmfdb} that having discriminant
$(7)^c$ with $c$ even is not sufficient to ensure that the Galois
group is even. 
 In fact, for each even value
of $c$, only half of the fields, counted by mass, have even Galois group.
We computed lists of congruences for each $7$-adic septic field with
even Galois group and then merged the lists of congruences.

To target a specific $7$-adic field, we start with a defining
polynomial such that a power basis formed from a root $\alpha$ will
generate the ring of integers over $\Z_7$.  We then consider a generic
element 
$\beta = a_0+a_1\alpha\pots a_6\alpha^6$ and
compute its characteristic polynomial in $\Z_7[a_0\cots a_6][x]$.
Working modulo $7^2$ we then enumerate all possibilities for the
polynomial.

\subsubsection{Savings from $\mbox{\rm ord}_3(G)=1$ at $p=3$}
Cases when $3$ is wildly ramified also offer an opportunity to
reduce the search time by more refined targeting.  As can be seen from
the two relevant lines of Table~\ref{wildtable}, the decomposition subgroup is 
isomorphic to $C_3$ or $S_3$.  In
either case, the orbit partition for the decomposition group is $(3,3,1)$.
Thus, a defining polynomial factors as the product of two cubics times
a linear polynomial over $\Z_3$.  

The savings comes from the fact that the two cubics have to
define the same $3$-adic field.  So, the procedure here starts
with computing possible polynomials for each ramified cubic
extension of $\Q_3$ modulo some $3^r$.  
We take all products of the form $(x+a)g_1g_2$ where the $g_i$ come
from the list for a given field and $a$ runs through all possibilities 
in $\Z/3^t$.  In our actual search, we worked modulo $3^2$.

The resulting local targets are considerably smaller.
For example, a target $((3,3,1),8)$ from our general method includes
cases where the two cubic factors define non-isomorphic
fields with discriminant ideal $(3)^4$ and also cases where the cubics have discriminant ideals
$(3)^3$ and $(3)^5$.  All these possibilities are not searched
in our refinement. 

\subsubsection{Exploiting arithmetic equivalence at $p=2$}
\label{arith-equiv}
The final refinement we use exploits the fact that for each octic
field sought, we need to find
just one of its two siblings in degree $7$.
These pairs of septic fields are examples of arithmetically equivalent
fields.  The two fields $K_{7a}$ and $K_{7b}$ have the same Dedekind zeta function, the same
discriminant, and the same ramification partition at all odd 
primes.   However, at $p=2$ one can have $\lambda_{7a} \neq \lambda_{7b}$.

In Table~\ref{wildtable}, there are two orbit partitions for the inertia
group $A_4$, and again two orbit partitions for $V$.  For each of
these cases, if a septic $\gseptic$ field has inertia subgroup $I$ and one orbit
partition, its sibling has the other orbit partition for $I$.
We save by 
targeting $43$ and $4111$, 
but not their transforms $61$ and $2221$.  

\subsubsection{Savings from global root numbers being $1$}
\label{signs}
As we mentioned in \Se\ref{hunter}, our code does not distinguish signatures.  
If it did, there would be an opportunity for 
yet further savings as follows.   A separable algebra
$K_v$ over $\Q_v$ has a local root number $\epsilon(K_v) \in \{1,i,-1,-i\}$.
For $v=\infty$, one has 
 $\epsilon(\R^r \C^s) = (-i)^s$.   For $v$ a prime $p$,
 one has $\epsilon(K_p) = 1$ unless the inertia group
 $I_p$ has even order.   Further information about
 local root numbers is at \cite[\S3.3]{jr-local-database},
 with many root numbers calculated on the associated database.
  The savings comes from the reciprocity relation $\prod_v \epsilon(K_v) = 1$,
  so that the signature is restricted by the behavior at ramifying 
  primes.  
  
  While we are not using local root numbers in our searches, we are using 
  them in our interpretation of the output of our first case.  
  Interesting facts here include the general
  formulas  $\epsilon(K_{7a,v}) = \epsilon(K_{7b,v})$ and 
  $\epsilon(K_{8,v}) = 1$.   Also, from Tables~\ref{tame7table} and
  \ref{wildtable}, 
  one has equality of 
  discriminant exponents $c_7=c_8$ at a prime $p$
  if and only if $|I_p|$ is odd; so in this case the septic
  sign $\epsilon(K_{7a,p}) = \epsilon(K_{7b,p})$ is $1$.  
\section{Results for $\g$ and related groups}
\label{summary}

\subsection{A complete list of $\g$ octics.}
Our search for $\g$ fields with $\rd(K) \leq 21$
took $41$ CPU-hours and confirmed that the discriminant $21^8$ 
given in \cite{kluner-malle} is indeed the smallest.  The extended
search through $\rd(K) \leq 30$ took approximately four
CPU-months.  
In this extended search, we combined targets for a given prime in a subsearch whenever the contribution to the
octic field discriminant is the same.  In this sense, the computation
consisted of $1471$ subsearches of varying difficulty.  The fastest
$380$ cases took at most $10$ seconds each, the median length case
took $6.5$ minutes, and the slowest ten cases took from $20$ to $35$
hours each.  The slowest cases all involved searches where 
$c_7=c_8$ for every ramifying prime.  This larger search 
found twenty-five fields.

\begin{thm}
\label{octicthm1}
  There are exactly $25$ octic fields with Galois group $\g$ and discriminant
  $\leq 30^8$.  The smallest discriminant of such a field is $21^8$
  and the field is given by
  \[ x^8 - 4x^7 + 7x^6 - 7x^5 + 7x^4 - 7x^3 + 7x^2 + 5x + 1 .\]
\end{thm}
\noindent The full list of fields 
is available in a computer-readable format by searching the website
\cite{jr-global-database}.

Figure~\ref{pict7and8} gives a visualization of how complete lists of $\g$ octics
and complete lists of $\gseptic$ septics have little to do with one another.   The 
$25$ $\g$ octics of Theorem~\ref{octicthm1}
give the $24$ points beneath the $\delta_8=30$ line, the drop of one coming from 
the fact that the point $(\delta_7,\delta_8) \approx (24.88,28.00)$ comes from
two fields.   Theorem~\ref{septicthm} likewise says that there are only $23$ points to the left of 
the $\delta_7=12$ line.   The rectangle where both $\delta_7 \leq 12$ and  $\delta_8 \leq 30$
contains just a single point.

Another way of making clear how the complete lists differ sharply is to consider 
first fields.  The first octic field, highlighted in Theorem~\ref{octicthm1},
 is a sibling of the famous Trinks field $\Q[x]/(x^7-7x+3)$ \cite{Tr}.  As
noted earlier, septic $\gseptic$ fields come in sibling pairs with the same
discriminant.  In the list of septic $\gseptic$ fields ordered by root
discriminant at \cite{jr-global-database}, the Trinks field is currently in the 
$1009$th pair, with root discriminant about $23.70$. 
Conversely, the octic sibling of the first septic $\gseptic$ pair currently
ranks $66$th in the corresponding list of octic fields
at \cite{jr-global-database}.  

\begin{figure}[htb]
\begin{center}
\includegraphics{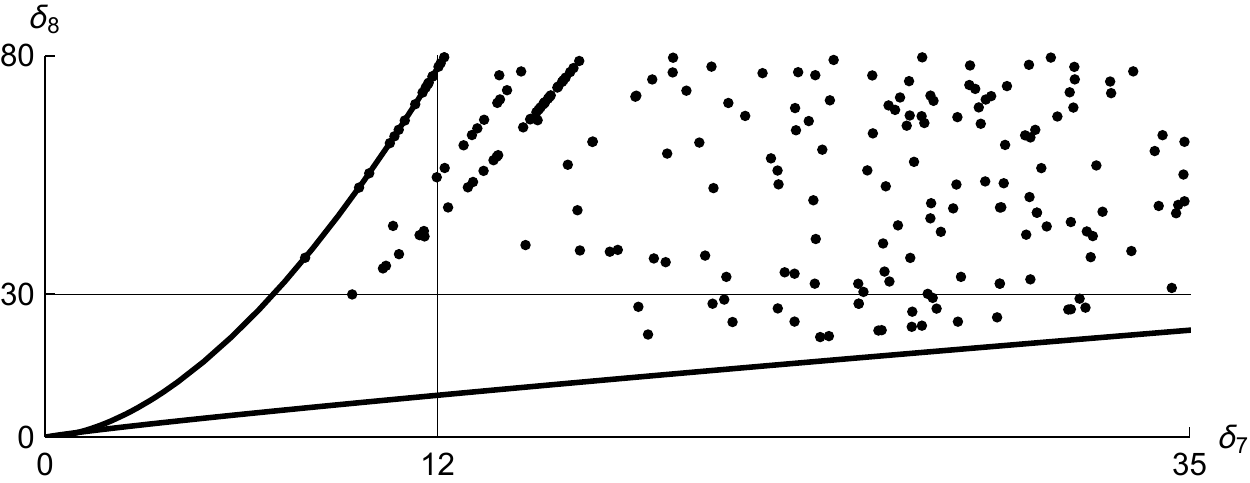} 
\caption{\label{pict7and8} Root discriminant pairs $(\delta_7,\delta_8)$ 
associated to $\GL_3(2) \cong \PSL_2(7)$, including all pairs with $\delta_7 \leq 12$ 
from Theorem~\ref{septicthm} and all pairs with $\delta_8 \leq 30$ from Theorem~\ref{octicthm1}.}
\end{center}
\end{figure}

Figure~\ref{pict7and8} also gives some sense of the
geography of sibling triples $(K_{7a},K_{7b},K_8)$. 
 The upper bound corresponds to the extreme case $D_8 = D_7^2$,
graphed as $\delta_8 = \delta_7^{7/4}$. 
The lower bound likewise corresponds to the extreme 
case $D_8 = D_7$, graphed as 
$\delta_8 = \delta_7^{7/8}$.   

As one example of  the
details visible in this geography, note that the figure
shows no point on the lower bound.    This is because 
a necessary and sufficient condition to be
on the lower bound is that 
all inertia groups have 
odd size, by Tables~\ref{tame7table} and \ref{wildtable}. 
As mentioned in \Se\ref{signs}, this condition forces the septic $p$-adic
local signs to all be $1$.  Reciprocity then forces
the infinite local sign to be $1$ as well, which
forces the fields to be totally real.   We
expect that the first instance of such a totally real sibling triple 
 is well outside the window of Figure~\ref{pict7and8}.  
There is just one such sibling triple currently
on the Kl\"uners-Malle database \cite{kluner-malle}, with $D_7 = D_8 = 2^4 13^4 131^4$,
and thus $(\delta_7,\delta_8) = (104.33,58.36)$.  
On the other hand, the figure shows twelve points seemingly
on a curve just above the lower bound.  These points
are the ones with $D_8/D_7 = 9$, and thus indeed 
lie on the curve $\delta_8 = 3^{1/4} \delta_7^{7/8}$.  
They come from the twelve fields having the prime $3$ in
bold on Table~\ref{fulltable}.

\subsection{Class groups and class fields} 
\label{classnumbers} The class groups of all twenty-five
fields in Theorem~\ref{octicthm1}
can be computed unconditionally by either {\em Magma} \cite{magma} or 
{\em Pari} \cite{gp}.   They are all cyclic and Table~\ref{fulltable}
gives their orders.    The fact that all but three of these class 
groups are non-trivial is already remarkable.  By way of contrast,  the first 
$620$ $S_5$ quintic fields ordered by 
absolute discriminant all have trivial class group. 

Using Theorem~\ref{octicthm1}, one can get complete
lists of fields with root discriminant $\leq 30$ for many other 
groups via class field theory.   We restrict ourselves 
to three corollaries, chosen because they interact interestingly
with the class groups on Table~\ref{fulltable}.  
The Galois groups involved in these corollaries are all even,
and so absolute discriminants coincide with discriminants.  

\begin{table}[htb]
\[\begin{array}{c|c|c|c|crcccc}
\textrm{rd}(K_8) & \textrm{rd}(K_7) & D_8 & D_7 & h & = & a & \ell & c & e \\
\hline
21.00 & 23.70 & {\bf3}^{8} 7^{8} & {\bf3}^{6} 7^{8} & 4 & = & 2 & 2 & \text{} & \text{} \\
21.21 & 23.97 & 2^{6} {\bf3}^{4} 53^{4} & 2^{6} {\bf3}^{2} 53^{4} & 4 & = & 2 & 2 & \text{} & \text{} \\
21.54 & 18.44 & 2^{16} 29^{4} & {\bf2}^{10} 29^{4} & 1 &   & \text{} & \text{} & \text{} & \text{} \\
22.37 & 25.48 & 2^{6} {\bf3}^{4} 59^{4} & 2^{6} {\bf3}^{2} 59^{4} & 4 & = & 2 & 2 & \text{} & \text{} \\
22.45 & 25.58 & {\bf3}^{6} 97^{4} & {\bf3}^{4} 97^{4} & 2 &  = & 2 & \text{} & \text{} & \text{} \\
23.16 & 26.50 & 2^{6} {\bf3}^{6} 11^{6} & 2^{6} {\bf3}^{4} 11^{6} & 6 & = & 2 & \text{} & 3 & \text{} \\
23.39 & 26.81 & {\bf3}^{4} 5^{4} 11^{6} & {\bf3}^{2} 5^{4} 11^{6} & 4 & = & 2 & 2 & \text{} & \text{} \\  
24.16 & 21.02 & 2^{16} 11^{6} & {\bf2}^{10} 11^{6} & 1 &   & \text{} & \text{} & \text{} & \text{} \\
24.23 & 27.91 & {\bf3}^{6} 113^{4} & {\bf3}^{4} 113^{4} & 2 &  = & 2 & \text{} & \text{} & \text{} \\ 
24.25 & 22.92 & 2^{8} {\bf3}^{4} 7^{8} & 2^{6} {\bf3}^{2} 7^{8} & 2 &   = & 2 & \text{} & \text{} & \text{} \\
25.14 & 29.11 & 2^{6} {\bf3}^{6} 43^{4} & 2^{6} {\bf3}^{4} 43^{4} & 2 &  = & 2 & \text{} & \text{} & \text{} \\
26.32 & 26.52 & 2^{6} {\bf5}^{4} 7^{8} & 2^{6} {\bf5}^{2} 7^{8} & 2 &  = & 2 & \text{} & \text{} & \text{} \\
26.78 & 31.29 & {\bf3}^{4} 239^{4} & {\bf3}^{2} 239^{4} & 4 &  = & 2 & 2 & \text{} & \text{} \\
26.84 & 31.37 & 2^{6} {\bf3}^{6} 7^{8} & 2^{6} {\bf3}^{4} 7^{8} & 2 &  = & 2 & \text{} & \text{} & \text{} \\
26.97 & 27.26 & 2^{6} {\bf5}^{6} 23^{4} & 2^{6} {\bf5}^{4} 23^{4} & 8 &  = & 4 & 2 & \text{} & \text{} \\
27.01 & 22.41 & 2^{8} {\bf5}^{4} 11^{6} & 2^{6} {\bf5}^{2} 11^{6} & 2 &  = & 2 & \text{} & \text{} & \text{} \\
27.17 & 31.82 & 2^{6} {\bf3}^{8} 29^{4} & 2^{6} {\bf3}^{6} 29^{4} & 4 & = & 2 & 2 & \text{} & \text{} \\
27.35 & 18.14 & 2^{8} {\bf11}^{4} 17^{4} & 2^{6} {\bf11}^{2} 17^{4} & 2 & = & 2 & \text{} & \text{} & \text{} \\  
28.00 & 20.41 & {\bf2}^{16} 7^{8} & {\bf2}^{8} 7^{8} & 2 &  = & 2 & \text{} & \text{} & \text{} \\
28.00 & 24.88 & 2^{16} 7^{8} & {\bf2}^{10} 7^{8} & 1 &  & \text{} & \text{} & \text{} & \text{} \\
28.00 & 24.88 & 2^{16} 7^{8} & {\bf2}^{10} 7^{8} &{2} &  = & \text{} & \text{} & \text{} & 2 \\
28.86 & 20.77 & 7^{8} {\bf17}^{4} & 7^{8} {\bf17}^{2} & 2 & = & 2 & \text{} & \text{} & \text{} \\ 
29.05 & 31.64 & {\bf2}^{8} 211^{4} & {\bf2}^{4} 211^{4} & 4 & = & 2 & 2 & \text{} & \text{} \\
29.22 & 27.14 & 2^{4} {\bf7}^{4} 61^{4} & 2^{4} {\bf7}^{2} 61^{4} & 4 &  = & 2 & 2 & \text{} & \text{} \\
29.94 & 9.39 & 2^{6} {\bf317}^{4} & 2^{6} {\bf317}^{2} & 6 &  = & 2 & \text{} & 3 & \text{} \\
\end{array}\]
\caption{\label{fulltable}  Discriminants of the
octic $\PSL_2(7)$ fields of Theorem~\ref{octicthm1} 
and their septic siblings, and also the class number
$h$ of each octic.  The boldface conventions
and the factorization $h = a \ell c e$ are explained in the paragraph
containing equation~\eqref{alc}.}
\end{table}

To start, consider the unramified tower $K_{32}/K_{16}/K_8$ coming
from the first field $K_8$ and its cyclic class group of order four.
Defining equations can be computed using {\em Magma} \cite{magma} or
{\em Pari}'s \cite{gp} class field theory commands.  Following the
conventions of \Se\ref{notation}, let $G_{16} = \Gal(K_{16})$ and
$G_{32} = \Gal(K_{32})$ be the corresponding Galois groups.  Then
$K_{16}$ and $K_{32}$ are the unique fields with smallest 
discriminant for these Galois groups.  In fact, $G_{16}$ is just the
Cartesian product $\PSL_2(7) \times C_2 = 16T714$.  More
interestingly, $G_{32} = 32T34620$ is a non-split double cover of
$G_{16}$, having $\SL_2(7) = 16T715$ as a subgroup with quotient group
$C_2$.  Similar analysis for all twenty-five base fields,
allowing ramified towers $K_{32}/K_{16}/K_8$ as well, gives the
following consequence of Theorem~\ref{octicthm1}.

\begin{cor}  There are exactly $25$ number fields with Galois group $\g \times C_2 = 16T714$ and
discriminant $\leq 30^{16}$.  The smallest discriminant of such a field is $21^{16}$
and the field has defining polynomial 
\begin{align*}
x^{16}&-4 x^{15}+9 x^{14}-14 x^{13}+14 x^{12}-14 x^{10}+8 x^9+45 x^8-82 x^7 \\ 
  &+49 x^6+63 x^5-112 x^4+49 x^3+99 x^2-130 x+100.
\end{align*}
There are exactly $14$ number fields with Galois group $32T34620$ and  
discriminant $\leq 30^{32}$.   The smallest discriminant of such a field is $21^{32}$
and the field has defining polynomial 
\begin{align*}
x^{32}&-x^{31}+2 x^{30}+x^{29}+8 x^{28}-7 x^{27}+21 x^{26}-9 x^{25}-12
    x^{24}+248 x^{23} \\ &-548 x^{22}-65 x^{21} +2653 x^{20}-4879 x^{19}+2564
    x^{18}+4198 x^{17}  -7780 x^{16} \\ &+3593 x^{15} +4020 x^{14}-7014 x^{13}+4935
    x^{12}-2042 x^{11}+929 x^{10}-787 x^9 \\ &+695 x^8-215 x^7+70 x^6-42 x^5+15
    x^4-15 x^3+2 x^2+x+1.
\end{align*}
\end{cor} 

As a next example, note that the sixth field $K_8$ has class number divisible by 
$3$, yielding a tower $K_{24}/K_8$.     Let $G_{24} = \Gal(K_{24})$.   To
find the field with smallest discriminant and Galois group $G_{24}$, 
we only have to look at perhaps-ramified abelian cubic extensions for the first six fields 
on the list.

The Galois group $G_{24}$ is in fact $24T284$, which is  $\g$ itself, but now in its action
on cosets of $C_7$.   So $\g$ octics are in bijection with $24T284$ fields 
via an elementary resolvent construction.  Inspecting the twenty-five $24T284$ fields
coming from the twenty-five octics gives the following result.

\begin{cor} There are exactly three number fields with Galois group $24T284$ and  discriminant
$\leq 30^{24}$.    The smallest discriminant of such a field is 
$(66^{3/4})^{22} \approx 23.16^{22}$ and the field has defining polynomial
\begin{align*}
x^{24}&-6 x^{23}+14 x^{22}-8 x^{21}-26 x^{20}+34 x^{19}+72 x^{18}-204
    x^{17} +109 x^{16} \\
    &+162 x^{15}  -148 x^{14}-260 x^{13}+496 x^{12}-248 x^{11}+18
    x^{10}-216 x^9 \\
    &+484 x^8 -402 x^7+156 x^6-74 x^5+102 x^4-76 x^3+22 x^2-2 x+1.
 \end{align*}
\end{cor}
\noindent The fields in question are respectively unramified, ramified, and unramified
extensions of the sixth, eighth, and twenty-fifth field on Table~\ref{fulltable}. 

   Our discussion so far has exhibited instances of three systematic contributors
to class groups of $\g$ octics $K=K_8$.   We call these the abelian, lifting,
and cubic contributions, and they are the sources of some of the numbers 
on Table~\ref{fulltable}:
\begin{align}
\label{alc}
a & = [AK:K], &
\ell  & \in \{1,2\}, &
c & \in \{1,3\}. 
\end{align} For the abelian
contribution, $A$ is the largest cyclotomic field such that $AK$ is unramified 
over $K$.  On the table, the primes for which $A/\Q$ is ramified
are put in bold in the $D_8$ column of Table~\ref{fulltable}.  To identify the lifting contribution,
we use the septic local signs $\epsilon_v$ of \Se\ref{signs}.   
A prime $p$ is put in bold
in the $D_7$ column exactly when $\epsilon_p=-1$.  If, 
for every such odd $p$, the inertia group
$I_p$ has order divisible by $2^{{\rm ord}_2(p-1)}$, 
and if also an analogous condition at $2$ holds if $2$ is in bold, 
then $KA$ has an unramified quadratic extension
$\widetilde{KA}$ with $\widetilde{KA}/A$ having Galois group
$\SL_2(7)$.  The lifting contribution is $\ell=2$ in this
case and otherwise $\ell=1$.  The cubic contribution
is $c=3$ if the canonical extension $K_{24}/K_8$ as above
is unramified, and $c=1$ if it is ramified.   

   In the general analysis of $\PSL_2(7)$ number fields $K$, let 
$e$ denote the rest of the class number $h$, meaning
$e = h/(a \ell c)$.  On the table, $e>1$ only once.  In this case, 
computation shows that the Hilbert class field $K_{16}$ 
has defining polynomial
\begin{equation}
\label{fifthfield}
x^{16}+40 x^{14}+588 x^{12}+3808 x^{10}+12236 x^8+9856 x^6+3248 x^4+384 x^2+16.
\end{equation}
Its Galois group $\Gal(K_{16}) = 16T1506$ has order
$2^4 |\PSL_2(7)|$.  The resulting field, with discriminant $28^{16}$, however is not minimal.

\begin{cor}  There are exactly five number fields with Galois group 
$16T1506$ and discriminant $\leq 30^{16}$.    The smallest
discriminant  $2^{16} 3^8 53^8 \approx 25.22^{16}$ arises from two fields.   These fields
have defining polynomials $f(x^2)$ and $f(-3x^2)$ 
where
\[
f(x) = x^8+8 x^7+32 x^6+44 x^5+382 x^4+496 x^3+656 x^2-20 x+1.
\]
\end{cor}
\noindent The next smallest discriminant, $2^{12} 3^{12} 53^8 \approx 27.91^{16}$, also arises
twice, and so the field defined by \eqref{fifthfield} is in fact last on the list.

\section{Results for $\PSL_2(5)$ and $\PGL_2(5)$}
\label{pgl2}

\subsection{Complete lists of $\PSL_2(5)$ and $\PGL_2(5)$ sextics} 
Our second illustration of the method of mixed degree searches comes 
from the sextic groups $\PSL_2(5) = 6T12$ and $\PGL_2(5)=6T14$. 
The fields with smallest root discriminant were obtained
by a direct sextic search in \cite{FP} and \cite{FPDH}. 
These root discriminants are $2^1 67^{1/3} \approx 8.12$ and
$2^1 3^{1/3} 7^{1/2} \approx 11.01$.  

In this second illustration, the smaller degree is $5$, via
the isomorphisms $A_5 \cong \PSL_2(5)$ and $S_5 \cong \PGL_2(5)$. 
Table~\ref{psltable} gives an analysis of tame ramification, with the bottom
three lines being relevant to $S_5 \cong \PGL_2(5)$ only.  As before, each partition
$\lambda_n$ determines the corresponding discriminant exponent
$c_n$ via formula~\eqref{tamec}.
\begin{table}[htb]
\caption{\label{psltable} Tame ramification data for $A_5 \cong \PSL_2(5)$ and $S_5 \cong \PGL_2(5)$ in degrees
$5$ and $6$.}
\[
\begin{array}{l|l||c|c}
\lambda_5 & \lambda_6 & c_5 & c_6 \\ \hline
5 & 51 & 4 & 4 \\
311 & 33 & 2 & 4 \\
221 & 2211 & 2 & 2 \\
\hline
41 & 411 & 3 & 3 \\
32 & 6 & 3 & 5 \\
2111 & 222 & 1 & 3
\end{array}
\]
\end{table}
The behavior of wild ramification in this 5-to-6 context is analogous
to the case of $\gseptic$ and $\g$  discussed earlier
and so we omit the detailed analysis.   From \cite{jr-tw}, one knows that the bounds suggested 
by the tame table hold in general: $D_5 \leq D_6 \leq D_5^2$ for 
$A_5 \cong \PSL_2(5)$ and $|D_5| \leq |D_6| \leq |D_5|^3$ for
$S_5 \cong \PGL_2(5)$.

As pointed out in \cite{jr-tw}, the first bound $|D_5| \leq |D_6|$
implies that a complete table of quintic fields up through discriminant 
bound $B$ determines the corresponding complete table of sextic
fields up through $B$.  In contrast to the situation for
$\gseptic \cong \g$, this observation gives non-empty lists
in the larger degree.    Taking $B=12\z000\z000$ from our extension \cite{jr-global-database} of 
\cite{SPD}, one gets $78$ $\PSL_2(5)$ sextics and $34$ $\PGL_2(5)$ sextics
with root discriminant at most $B^{1/6} \approx 15.13$.

On Table~\ref{psltable} there are three instances when $c_5<c_6$.  
Accordingly, we can substantially reduce the quintic search space
by targeting.  The result for root discriminant $\delta_6 \leq 35$ is
as follows.
\begin{thm} 
\label{sexticthm}
Among sextic fields with absolute discriminant $\leq 35^6$, exactly $2361$ 
have Galois group $\PSL_2(5)$ and $3454$ have Galois group $\PGL_2(5)$.
\end{thm}
\noindent Lists of fields can be retrieved by searching the website for \cite{jr-global-database}.

In parallel with the figure for our first case, Figure~\ref{pict5and6}
illustrates our second case.   The regularity near the bottom boundary
$\delta_6 = \delta_5^{5/6}$ is easily explained, as follows.   For any pair $(K_5,K_6)$,
the ratio $D_6/D_5$ is always a perfect square $r^2$.   The pair gives rise 
to a point on the curve $\delta_6 = r^{1/3} \delta_5^{5/6}$.   
In the $S_5 \cong \PGL_2(5)$ case, the curves corresponding to 
$r=1,2,\dots,14,15$ are all clearly visible.  The first ``missing curve,''
clearly visible as a gap,  
corresponds to $r=16=2^4$.  This curve is missing
 because none of the $2$-adic possibilities for $(c_5,c_6)$ satisfy $c_6-c_5=8$.
In the $A_5 \cong \PSL_2(5)$ case, there are fewer $2$-adic possibilities
and the first four visible gaps correspond to $r=4$, $8$, $12$, and $16$.

\begin{figure}[htb]
\begin{center}
\includegraphics[width=4.8in]{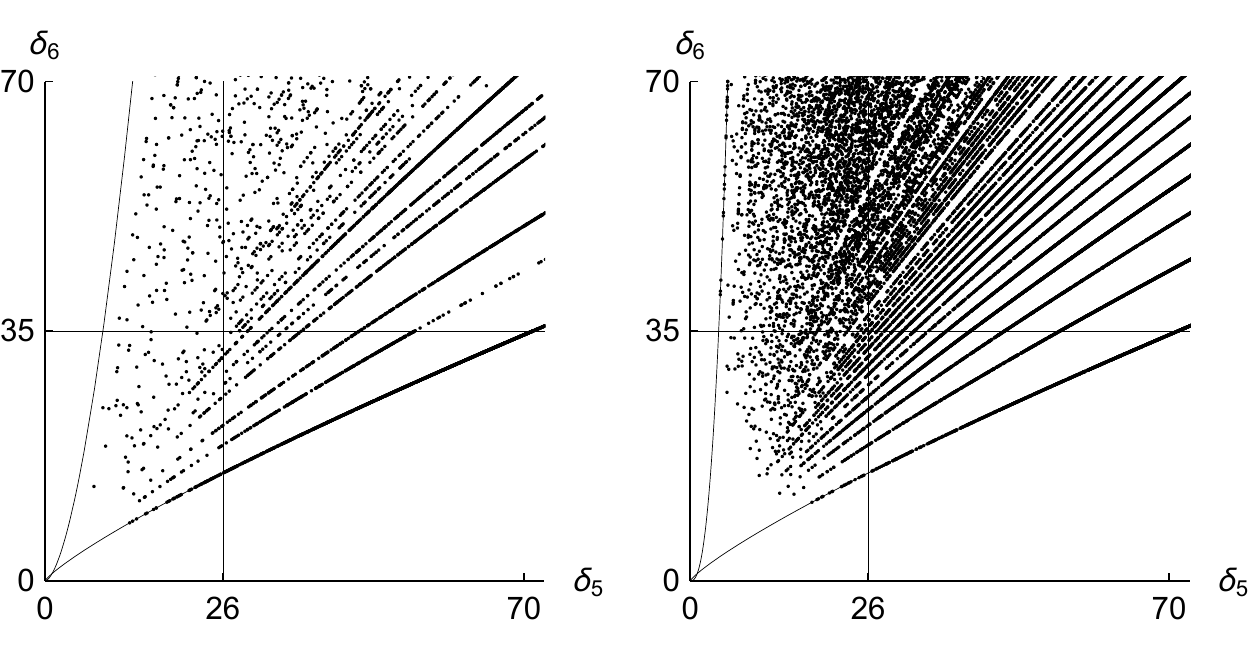} 
\end{center}
\caption{\label{pict5and6} Root discriminant pairs $(\delta_5,\delta_6)$ 
associated to $A_5 \cong \PSL_2(5)$ (left) and $S_5 \cong \PGL_2(5)$ (right), 
including all pairs in the window with $\delta_5 \leq 26$ from \cite{jr-global-database}, and all pairs 
with $\delta_6 \leq 35$ from Theorem~\ref{sexticthm}.}
\end{figure}

\subsection{Connections with expected mass formulas}  
Let $NF(G,x)$ denote the set of isomorphism classes 
of number fields $K$ with $\Gal(K)=G$ and root discriminant 
at most $x$.  From the quintic search in \cite{jr-global-database}, one knows 
\begin{align}
\label{bigquintic}
|NF(A_5,26)| &= 539, &  |NF(S_5,26)| & = 726862.
\end{align}
 The ratio
$539/726862 \approx 0.00074$ is an instance of the familiar informal principle
``$S_n$ fields are common
but $A_n$ fields are rare."   In this light, the much larger 
ratio $2361/3454 \approx 0.68356$ from Theorem~\ref{sexticthm} is surprising.  

However, the fact that $NF(\PSL_2(5),35)$ and $NF(\PGL_2(5),35)$ have
such similar sizes can be explained as follows.   For $g \in S_n$,
let $o_g$ be its number of cycles and let $\epsilon_g = n-o_g$.  
For a transitive permutation group
$G \subseteq S_n$ and a conjugacy class $C \subseteq G$, 
define $\epsilon_C$ to be $\epsilon_g$ for any $g \in C$.  
Define $a_G$ to be the minimum of the $\epsilon_C$ and
$b_G$ to be the number of classes obtaining this minimum.
Then Malle conjectures an asymptotic growth rate
\[
|NF(G,x)| \sim c_G x^{n/a_G} \log(x)^{b_G-1},
\]
for some constant $c_G$ \cite{Ma}.   For $A_n$, the two minimizing classes have
cycle types $2^2 1^{n-4}$ and $3 1^{n-3}$, while for 
$S_n$ the unique minimizing class is $21^{n-2}$.  
Thus, consistent with numerical data like \eqref{bigquintic},
one expects very different growth rates:
\begin{align}
|NF(A_n,x)| & \sim c_{A_n} x^{n/2} \log x,& 
|NF(S_n,x)| & \sim c_{S_n} x^n.
\end{align}
For $n \leq 5$, this growth rate is proved for $S_n$ with identified constants, 
and it is known that the growth for $A_n$ is indeed slower; see  \cite{Bh5} for 
$S_5$ and \cite{BCT} for $A_5$.  
 
But now, for $G \in \{\PSL_2(5),\PGL_2(5)\}$, the unique
minimizing class is $2^2 1^2$.  Thus, consistent with 
Theorem~\ref{sexticthm}, these two groups should have the
same asymptotics up to a constant, both having the form 
$|NF(G,x)| \sim c_G x^{3}$.   There 
are similar comparisons associated to 
our two other theorems.   For the group of Theorem~\ref{octicthm1},
one should have
$|NF(\PSL_2(7),x)| \sim c_{\PSL_2(7)} x^2 \log x$ from 
$2^4$ and $3^2 1^2$.  This growth rate is substantially
less than the expected $|NF(\PGL_2(7),x)| \sim c_{\PGL_2(7)} x^{8/3}$
from $2^3 1^2$.    Indeed, there are eighteen
known $\PGL_2(7)$ octics with root discriminant $\leq 21$ \cite{jr-global-database}.  
For the groups of Theorem~\ref{septicthm}, $|NF(\gseptic,x)| \sim c_{\gseptic} x^{7/2}$ should 
grow more slowly than $|NF(A_7,x)| \sim c_{A_7} x^{7/2} \log x$.  
The search underlying Theorem~\ref{septicthm} is complete
through discriminant $12^7$.  It can be expected
to be near-complete for $x$ substantially past $12^7$, 
and these extra fields do indicate a general increase in $|NF(A_7,x)|/|NF(\gseptic,x)|$.

\subsection{Complete lists of Artin representations}
Theorems~\ref{octicthm1} and \ref{sexticthm} can each be viewed from a 
different perspective.  Computing octic
$\g$ fields is essentially the same as computing Artin representations
for the irreducible degree $7$ character of $\gseptic$; computing sextic
$\PSL_2(5)$ and $\PGL_2(5)$ fields is equivalent to computing Artin
representations for certain irreducible degree $5$ characters of $A_5$ and
$S_5$.  In each case, discriminants match conductors.

By targeting differently, one could get  lists of 
other Artin representations belonging to these and other small
groups, which are complete up through some conductor cutoff.   For example, $\g$ and $\PSL_2(5)$ 
each have two three-dimensional
representations.    These representations are particularly 
interesting because of their low degree, which means that the
associated $L$-functions are relatively accessible to analytic computations.  
We plan to draw up these lists elsewhere,
as they fit into the program of \cite{lmfdb} of systematically
cataloging $L$-functions.

\section{Results for $\gseptic$, $A_7$, and related groups}
\label{standard}

\subsection{Extending the known lists for $\gseptic$ and $A_7$}
The first pair of fields for $\gseptic$ and first
field for $A_7$
were determined by Kl\"uners and Malle in \cite{kluner-malle}.  
We extend the complete list of such fields by employing the same technique as
\cite{kluner-malle}: a standard
Hunter search modified to select polynomials with Galois group
contained in $A_7$.  The latter condition comes into play as the first
step when inspecting a polynomial in the search to see if it is
suitable.  Testing if the polynomial discriminant is a square can be
done very quickly and filters out all the polynomials 
with odd Galois group.  

\begin{thm}
\label{septicthm}
  Among septic fields with discriminant $\leq 12^7$, exactly
  $46$ have Galois group $\gseptic$ and 
  $17$ have Galois group $A_7$.
\end{thm}
\noindent Carrying out the computation up to discriminant $12^7$ took six and a
half months of CPU time and inspected roughly $10^{12}$ polynomials.
As in other cases, defining polynomials and other information for these fields can be
obtained from the website for \cite{jr-global-database}.

The computation establishing Theorem~\ref{septicthm} sheds light on 
the list of $\g$ fields established by Theorem~\ref{octicthm1}.
  The runtime of a Hunter search in degree $n$
with discriminant bound $B$ is proportional to $B^{(n+2)/4}$ \cite{jr-timing}.  Using
this, our estimate for the runtime for confirming the first $\g$ octic field 
by computing septic fields without targeting is approximately
$(21^8/12^7)^{9/4} 6.5/12 \approx 3$ million CPU-years.
To get Theorem~\ref{octicthm1}'s complete list through discriminant $30^8$ would then take $(30^8/12^7)^{9/4} 6.5/12 \approx 2$ billion CPU-years.

\subsection{The first $\gaffine$ field}
As observed in \cite{kluner-malle}, a sufficiently long complete list of septic $\gseptic$
fields can be used to determine the first octic $\gaffine$ field.
The splitting field of a $\gaffine$ polynomial contains (up to isomorphism),
two subfields $K_{7a}$ and $K_{7b}$ of degree $7$, two $K_{14a}$ 
and $K_{14b}$ of degree $14$ and Galois group $14T34$, and two
subfields $K_{8a}$ and $K_{8b}$ of degree $8$ and Galois group
$8T48$.  One of the septic fields is contained in both $K_{14a}$ and
$K_{14b}$, and the other is contained in neither.  

Since the septic fields are arithmetically equivalent, they have the 
same discriminants, i.e., $D_{7a}=D_{7b}$.  The other indices can be
adjusted so that 
\begin{equation} \label{disceq}
D_{7x}D_{8x}=D_{14x} \quad \text{ for } x\in\{a,b\}.
\end{equation}
This comes from a character relation on the relevant permutation
characters.   Because of the asymmetry in the field inclusions
described above, the fact that $\gseptic$ fields come in
arithmetically equivalent pairs does not play a role in our computations,
and so we have forty-six septic ground fields to consider
separately.   Accordingly, we drop $x$ from notation, always taking
the correct octic resolvent so that \eqref{disceq} holds.   

Of the forty-six septic $\gseptic$ fields with discriminant $\leq 12^7$, only
one has non-trivial narrow class group. 
This field has narrow class
number two, and is $K_{7} = \Q[x]/f(x)$ with 
\begin{equation}
\label{deg7}
f(x) = x^7 - x^6 - x^5 - 2 x^4 - 7 x^3 - x^2 + 3 x + 1.
\end{equation}
The unramified quadratic extension turns out to be simply $K_{14} = \Q[x]/f(-x^2)$,
which has Galois group $14T34$.   So $K_{14}$ and $K_{7}$ both
have root discriminant $5717^{2/7} \approx 11.84$.  By \eqref{disceq},
the sibling $K_{8}$ of $K_{14}$ has the even smaller root discriminant 
$5717^{1/4} \approx 8.70$.  

In general, given all septic $\gseptic$ fields up to some discriminant bound $B$, 
one can get all $14T34$ fields up to the discriminant bound $B^2$ 
via quadratic extensions.    If $\frakd$ is the
relative discriminant of $K_{14}/K_7$, then $D_{14}=D_7^2 N_{K_{7}/\Q}(\frakd)$, and so
$N_{K_{7}/\Q}(\frakd)$ must be at most $B^2/D^2_7$.   In the same way, the stronger bound 
$N_{K_{7}/\Q}(\frakd) \leq B/D_7$ is necessary and sufficient for 
 the resolvent $8T48$ field to have discriminant $\leq B$.
 Taking $B=12^7$ now, the quotient $B/D_7$ decreases
from $12^7/(13^2 109^2) \approx 17.85$ for the first two ground fields $K_7$
 to $12^7/(2^6 743^2) \approx 1.01$ for the last two ground fields.
For the first twenty-six ground fields, computation shows
that there are no $14T34$ overfields satisfying $N_{K_{7}/\Q}(\frakd) \leq B^2/D^2_7$. For the last twenty fields,
already $B/D_7<\sqrt{2}$ and so the lack of overfields, except for the unramified
one above, follows from the narrow class numbers being $1$.   Hence we have
the following corollary of Theorem~\ref{septicthm}.

\begin{cor}
\label{octicthm2}
  The field $\Q[x]/f(-x^2)$ from \eqref{deg7} is the only degree fourteen
  field with Galois group $14T34$ and discriminant $\leq 12^{14}$.  
  Its sibling, with defining polynomial
  \[ x^8 - 4x^7 + 8x^6 - 9x^5 + 7x^4 - 4x^3 + 2x^2 + 1,\]
  is the only the octic field with Galois group $\gaffine = 8T48$ and discriminant
  $\leq 12^7$.  These two fields have discriminants $5717^4 \approx 11.84^{14}$ and 
  $5717^2 \approx  11.84^7 \approx 8.70^{8}$ respectively.  
\end{cor}

\bibliographystyle{amsalpha}
\bibliography{jr}

\end{document}